\documentclass[11pt]{article}

\usepackage{amssymb, amsmath, amsthm, graphicx}
\usepackage[left=1in,top=1in,right=1in]{geometry}
\usepackage{subfigure}
\usepackage{ifthen}
\usepackage{comment} 
 
\usepackage[color=Orange!50!white, textwidth=22mm]{todonotes}

\date{}

\theoremstyle{plain}
      \newtheorem{theorem}{Theorem}[section]
      \newtheorem{lemma}[theorem]{Lemma}
      \newtheorem{claim}[theorem]{Claim}

      \newtheorem{proposition}[theorem]{Proposition}
      \newtheorem{conjecture}[theorem]{Conjecture}
\theoremstyle{definition}

\theoremstyle{remark}

\def\cn{\mbox{\rm cr}}

\title{Immersions and Albertson's conjecture }
\author{Jacob Fox\thanks{Stanford University, Stanford, CA. Supported by NSF awards DMS-2452737 and DMS-2154129. Email: {\tt jacobfox@stanford.edu.}} \and J\'anos Pach\thanks{Supported by NKFIH grant K-131529 and ERC Advanced Grant 882971 ``GeoScape.'' Email:
{\tt pach@cims.nyu.edu}.}\and  Andrew Suk\thanks{Department of Mathematics, University of California at San Diego, La Jolla, CA, 92093 USA. Supported by NSF grant DMS-2246847. Email: {\tt asuk@ucsd.edu}.} }

\begin{document}

\maketitle

\begin{abstract}

A graph is said to contain $K_k$ (a clique of size $k$) as a \emph{weak immersion} if it has $k$ vertices, pairwise connected by edge-disjoint paths. In 1989, Lescure and Meyniel made the following conjecture related to Hadwiger's conjecture: Every graph of chromatic number $k$ contains $K_k$ as a weak immersion. We prove this conjecture for graphs with at most $(1.64-o(1))k$ vertices. As an application, we make some progress on Albertson's conjecture, according to which every graph $G$ with chromatic number $k$ satisfies $\cn(G) \geq \cn(K_k)$. In particular, we show that the conjecture is true for all graphs of chromatic number $k$, provided that they have at most $(1.64-o(1))k$ vertices.  
\end{abstract}

\section{Introduction}

There are several famous problems in graph theory which state that over all graphs of a given chromatic number, some graph parameter is minimized by a complete graph. 
Obviously, the chromatic number of a graph is at least its clique number. The converse is false, but partial converses have been of central interest in graph theory. Hadwiger's conjecture states that every graph of chromatic number $k$ contains a $K_k$-minor. Wagner proved in 1937 that the case $k=5$ is equivalent to the four color theorem. Hadwiger's conjecture was verified for $k \leq 6$ by Robertson, Seymour, and Thomas \cite{RST}, and is open for $k\geq 7$. In the 1980's, Kostochka \cite{K84} and Thomason \cite{T84} proved that every graph of chromatic number $k$ contains a $K_t$-minor, where $t = \Omega(k/\sqrt{\log k})$,  which was improved to $t = \Omega(k/(\log k)^{1/4 + \epsilon})$ by Norin, Postle, and Song \cite{NPS}, and very recently, to $t = \Omega(k/\log\log k)$ by Delcourt and Postle \cite{DP}.   

In 1961, Haj\'os conjectured the following strengthening of Hadwiger's conjecture: Every graph of chromatic number $k$ contains a \emph{subdivision} of the complete graph $K_k$, i.e., it has $k$ so-called ``branch vertices'' connected by $\binom{k}{2}$ internally \emph{vertex-disjoint} paths.
Haj\'os' conjecture is true for $k \leq 4$, but for $k\ge 7$ it was disproved by Catlin \cite{Cat}. In fact, Erd\H{o}s and Fajtlowicz \cite{EF81} showed that almost all graphs are counterexamples (see also \cite{FLS}). The conjecture remains open for $k = 5,6$.  

Lescure and Meyniel \cite{LM} suggested a conjecture weaker than Haj\'os', which may still be true for every $k$. Instead of requiring that $G$ contains a subdivision of $K_k$, they wanted to prove the existence of $k$ branch vertices connected by ${k\choose 2}$ \emph{edge-disjoint} paths. Moreover, these paths may pass through some branch vertices other than their endpoints.  
They called such a subgraph of $G$ a \emph{weak immersion} of $K_k$.

 More precisely, a graph $G$ contains $H$ as a {weak immersion} if there is a mapping  $\phi$ from $V(H)\cup E(H)$, which maps each vertex of $H$ to a vertex in $G$ and each edge of $H$ to a path in $G$ such that

\begin{enumerate}

\item $\phi(u)\neq \phi(v)$, for distinct vertices $u,v \in V(H)$;

\item for distinct edges $e,f \in E(H)$, the paths $\phi(e)$ and $\phi(f)$ are edge-disjoint; and

\item for each edge $e=uv \in E(H)$, $\phi(e)$ is a path in $G$ with endpoints $\phi(u)$ and $\phi(v)$.
\end{enumerate}
\noindent If the following condition is also satisfied, we say that $G$ contains $H$ as a \emph{strong immersion}.

\begin{enumerate}
  \setcounter{enumi}{3}
    \item For each edge $e  \in E(H)$, the path $\phi(e)$ intersects the set of branch vertices, $\phi(V(H))$, only at its endpoints.
\end{enumerate}

\noindent  In 1989, Lescure and Meyniel conjectured the following.

\begin{conjecture}[\cite{LM}]\label{c1}
Every graph with chromatic number $k$ contains a weak immersion of the complete graph $K_k$.
\end{conjecture}

\noindent For $k\ge 3,$ the Lescure-Meyniel conjecture is an immediate corollary of Haj\'os' conjecture. DeVos, Kawarabayashi, Mohar, and Okamura \cite{De} verified Conjecture \ref{c1}  for $4 \leq k \leq 6$. Conjecture \ref{c1} remains open for $k \geq 7$.  According to a result of Gauthier, Le, and Wollan \cite{GLW}, every graph with chromatic number $k$ contains a weak immersion of $K_t$, where $t = (k-4)/3.54$ (see also \cite{DeD,Dv} for earlier bounds).  

Our first result shows that the Lescure-Meyniel conjecture is true for graphs whose number of vertices is not much larger than its clique number.

\begin{theorem}\label{main1}
Let $G$ be a graph with chromatic number $k$ and $n$ vertices. 
\begin{itemize} 
\item[(i)] If $n<1.4k-0.6$, then $G$ contains a weak immersion of the complete graph $K_k$. 
\item[(ii)] For each $\varepsilon>0$ and $k$ sufficiently large, if $n<(1.64-\varepsilon)k$, then $G$ contains a weak immersion of the complete graph $K_k$. 
\end{itemize}
\end{theorem} 

Note that the first part of Theorem \ref{main1} holds for all $k$, while the second part of Theorem \ref{main1} is better for sufficiently large $k$. The proof of the first part of Theorem \ref{main1} is much simpler and also illustrative of the general approach. 

%\janos{I think it would be better to have a two-part theorem: part (i) would be the old one, true for all k, and part (ii) the new one. We would spell out the proof of part (i), as in the earlier version, and then switch to the modifications required for part (ii).}

%With a much simpler proof using the same framework, we show that every graph with chromatic number $k$ and with at most $1.4(k-1)$ vertices contains $K_k$ as a weak  immersion. 

As an application, we use Theorem \ref{main1} to obtain new bounds on an old conjecture of Albertson.

The \emph{crossing number} of a graph $G$, $\cn(G)$, is the smallest number of edge crossings in any drawing of $G$ in the plane.  In 2007, Albertson conjectured the following.

\begin{conjecture}\label{c2}
    Every graph $G$ with chromatic number $k$ satisfies $\cn(G) \geq \cn(K_k)$.  
\end{conjecture}

Clearly, Albertson's conjecture is weaker than Haj\'os' conjecture.  Conjecture \ref{c2} vacuously holds for $k \leq 4$, since $\cn(K_4) = 0$. For $k = 5$, it is equivalent to the four color theorem. Following a sequence of results \cite{ACF,BT,A}, Albertson's conjecture has been confirmed for $k \leq 18$, but the problem is open for $k\geq 19$.  

A graph $G$ is said to be \emph{$k$-critical} if $\chi(G) = k$, and every proper subgraph of $G$ has chromatic number less than $k$. A $1$-critical graph is just a graph consisting of a single vertex. As $\cn(G)\geq \cn(H)$ holds for all subgraphs $H\subset G$, it suffices to prove Albertson's conjecture for $k$-critical graphs.  In \cite{BT}, Bar\'at and T\'oth verified Conjecture \ref{c2} for all $k$-critical graphs on at most $k + 4$ vertices, and Ackerman~\cite{A} proved the conjecture for all $k$-critical graphs with at least $3.03k$ vertices.  Our next result is the following.

\begin{theorem}\label{main2}  For every $\varepsilon>0$, there exists a sufficiently large integer $k=k(\varepsilon)$ with the following property. For every graph $G$ with $n\leq (1.64 - \varepsilon)k$ vertices and chromatic number $k$, we have $\cn(G) \geq \cn(K_k)$.
\end{theorem}

Our proof of Theorem \ref{main2} is based on Theorem \ref{main1}. It shows that any improvement of the constant factor in Theorem \ref{main1} would lead to the same improvement of the constant factor in Theorem \ref{main2}. 
\smallskip

\noindent\textbf{Organization.} We start Section~\ref{sec2} by establishing Theorem~\ref{main1}(i). The proof uses a classic result of Gallai~\cite{Ga} which characterizes $k$-critical graphs with at most $2k-2$ vertices. We will also apply an old theorem of Shannon~\cite{Shannon} about the chromatic index of multigraphs. The proof of Theorem~\ref{main1}(ii) follows the same idea, but now we have to be more careful with the construction of the multigraph whose chromatic index needs to be estimated. In addition to Shannon's theorem, we use another upper bound on the chromatic index, due to Vizing and Gupta~\cite{V,G}. The skeleton of the proof of Theorem~\ref{main1}(ii) is presented at the end of Section~\ref{sec2}, but the proof of Lemma~\ref{lem:fix}, the most important technical step leading to the improvement, will be postponed to Section~\ref{sec3}.  Section~\ref{sec4} is devoted to the proof of Theorem~\ref{main2}.

\section{Weak immersion}\label{sec2}

In this section, we prove Theorem \ref{main1}.  We begin by focusing on the proof of part (i), and then discuss how the proof can be modified to obtain part (ii). 

Given a graph $G = (V,E)$ and vertex $v \in V$, let $N_G(v)$ denote the degree of vertex $v$ in $G$.  If it is clear which graph we consider, we will drop the index $G$ and simply write $N(v)$. If we allow that two vertices of $G$ are connected by several edges, then $G$ is called a \emph{multigraph}. The number of edges connecting two given vertices is called the \emph{multiplicity} of the edge.

A classic result due to Gallai states that if $G$ is a $k$-critical graph on $n$ vertices, where $n \leq 2k-2$, then the complement of $G$ is disconnected.  This immediately implies the following.

\begin{lemma}[\cite{Ga}]\label{lemGa}
Let $k,n$ be positive integers with $n \leq 2k-2$.  If $G$ is a $k$-critical graph on $n$ vertices, then there is a vertex partition 
\[V(G) = V_1\cup V_2\cup \cdots \cup V_t,\]

\noindent where $t\geq 2$, such that $V_i$ is complete to $V_j$, for $i\neq j$, $|V_i| = n_i$, and the induced subgraph $G[V_i]$ is $k_i$-critical with $n_i \geq 2k_i - 1$. 
\end{lemma}

The {\it chromatic index} $\chi'(H)$ of a multigraph $H$ without loops is the minimum number of colors needed to properly color the edges of $H$, i.e., to color them in such a way that no two edges that share a vertex receive the same color.  We will need the following well known results about the chromatic index of graphs.

\begin{lemma}\label{vizing} Let $H$ be a multigraph without loops, with maximum degree $\Delta$ and edge multiplicity at most $\mu$. Then the chromatic index of $H$ satisfies\\
(i) (Shannon~\cite{Shannon})
$\chi'(H) \leq 3\Delta/2;$ \\
(ii) (Gupta~\cite{G}, Vizing~\cite{V})
$\chi'(H) \leq \Delta + \mu.$
\end{lemma}

\begin{proof}[Proof of Theorem \ref{main1}(i)]  We may assume that $k\geq 7$, since otherwise we obtain a weak immersion of $K_k$ by \cite{GLW}.  By possibly deleting vertices and edges, we may assume without loss of generality that $G$ is $k$-critical. Let $n=|V(G)|$ so $n \leq 1.4(k-1) \leq 2k-2$. By Lemma \ref{lemGa}, there is a vertex partition $V(G)=V_1 \cup \cdots \cup V_t$, where $t\geq 2$, such that $V_i$ is complete to $V_j$ for each $i \not = j$, and the induced subgraph $G[V_i]$ is $k_i$-critical for each $i$ with $n_i \geq 2k_i-1$ vertices.
Hence, $n= \sum_{i=1}^t n_i$ and $k=\sum_{i=1}^t k_i$. For each $i$, 
arbitrarily partition $V_i = U_i \cup W_i$ with $|U_i| = k_i$, so $|W_i| = n_i - k_i$. Let $U= \bigcup_i U_i$. 

In what follows, we will construct a weak immersion of $K_k$ with $U$ being the set of branch vertices.  Moreover, we will use all edges in $U$ as paths of length one in the weak immersion. By the Gallai decomposition, each nonadjacent pair of vertices in $U$ has both of its vertices in $U_i$ for some $i$. For each $i$ and vertex $u \in U_i$, let $f_u$ be a one-to-one function from the set of non-neighbors of $u$ in $U_i \setminus \{u\}$ to the set of neighbors of $u$ in $W_i$. Such a function $f_u$ exists as the degree of $u$ in $G[V_i]$ is at least $k_i-1$, as $G[V_i]$ is $k_i$-critical. If a nonadjacent pair $(u,u')$ of vertices in $U_i$ satisfies $f_u(u’)=f_{u'}(u)$, then we connect $u$ and $u’$ in the weak immersion by the path of length two with middle vertex $f_u(u’)$. Moreover, these paths will be edge-disjoint as $f_u$ is one-to-one.  See Figure~\ref{figpathu1}.

So far we have constructed edge-disjoint paths (which are of length one or two) connecting some pairs of branch vertices. We next describe how we connect the remaining pairs of vertices in $U$ by paths, which will each be of length four. For a pair $(u,u')$ of nonadjacent vertices in $U_i$ with $f_u(u') \not = f_{u'}(u)$, we will pick a vertex $t \in U \setminus U_i$ and the path of length four connecting $u$ to $u'$ will have the vertices in order as $u,f_u(u'),t,f_{u'}(u),u'$. We next describe how to pick the vertex $t=t(u,u')$ for each such pair $u,u'$ of nonadjacent vertices in the same $U_i$ with $f_u(u') \not = f_{u'}(u)$.  See Figure \ref{figpathu2}.

 \begin{figure}
  \centering
  \subfigure[{$f_u(u') = f_{u'}(u)$.}]{\label{figpathu1}\includegraphics[width=0.169\textwidth]{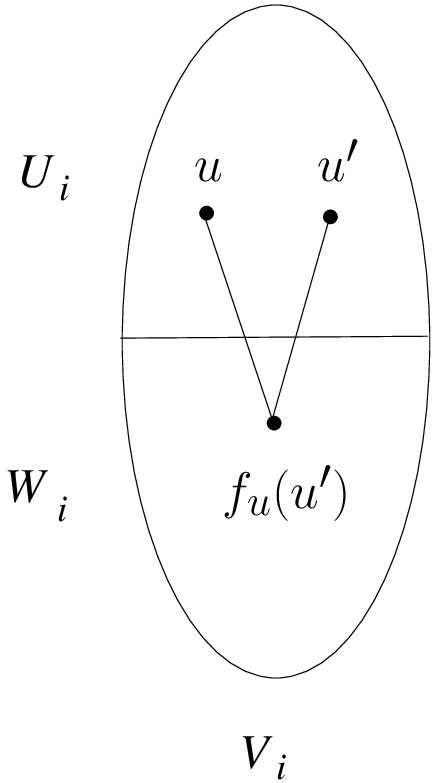}}\hspace{3cm}
    \subfigure[$f_u(u') \neq f_{u'}(u)$]{\label{figpathu2}\includegraphics[width=0.48\textwidth]{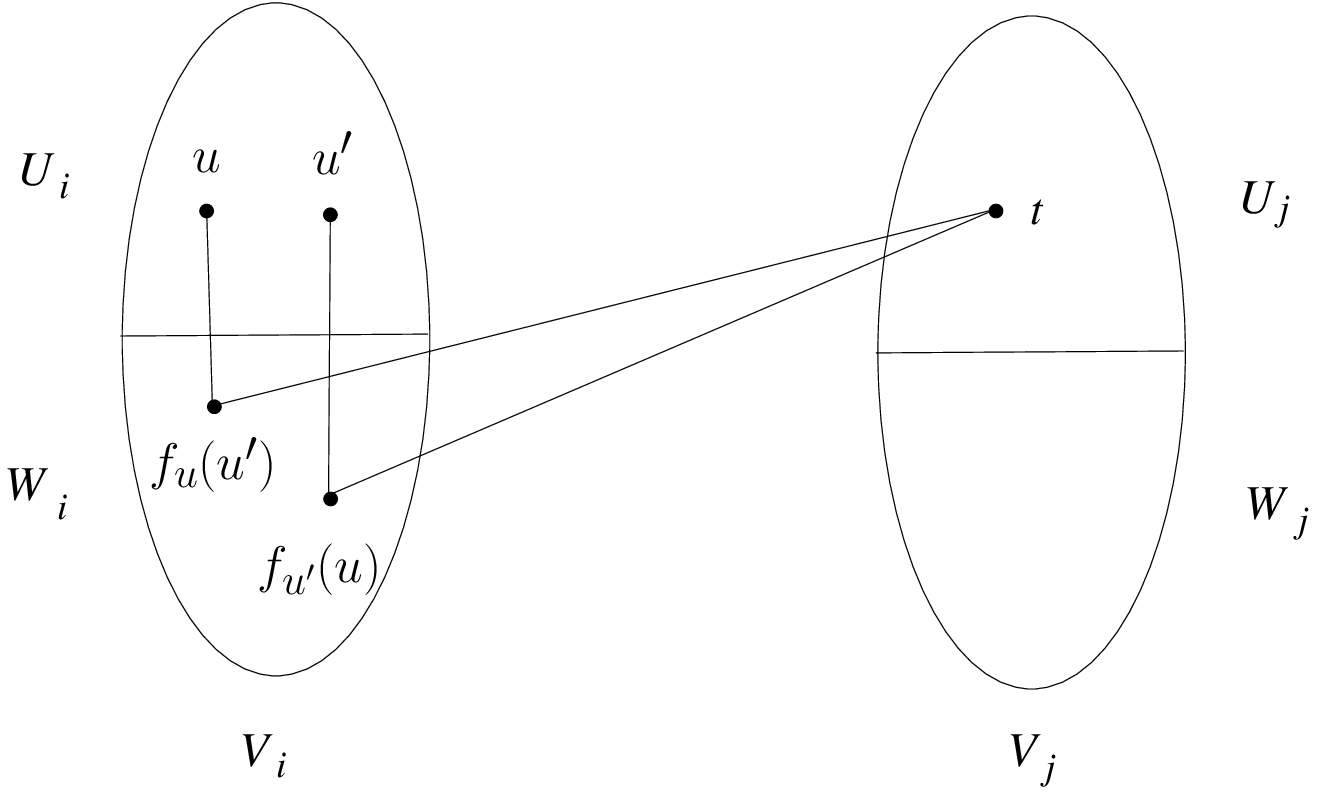}}
 \caption{Constructing a path from $u$ to $u'$.}
\end{figure}

Make an auxiliary multigraph $H_i$ on $W_i$ as follows. For each nonadjacent pair $(u,u’)$ with $u,u' \in U_i$ and $f_u(u') \not = f_{u'}(u)$, we add an edge between $f_u(u')$ and $f_{u'}(u)$ in $H_i$.  Clearly, $H_i$ does not contain loops as we require $f_u(u') \not = f_{u'}(u)$.  Since $f_u$ is one-to-one, the maximum degree in $H_i$ is at most $|U_i|=k_i$. Being able to pick the desired vertex $t = t(u,u') \in U \setminus U_i$ for each nonadjacent pair $u,u' \in U_i$ with $f_u(u') \not = f_{u'}(u)$, in order to obtain the desired paths of length four for the immersion, is equivalent to being able to properly color the edges of $H_i$ with color set $U \setminus U_i$.  As $n_i \geq 2k_i  - 1$, we have

\begin{equation}\label{firstinequality} k_i \leq n_i-k_i+1 = |W_i| + 1 \leq n-k+1 \leq 0.4(k-1),\end{equation}

\noindent where the last inequality follows from the fact that $n \leq 1.4(k-1).$  This implies that

\[|U\setminus U_i| = k - k_i \geq \frac35k.\]

\noindent On the other hand, by Lemma~\ref{vizing}(i), we have 
\begin{equation}\label{chiindex} \chi'(H_i) \leq 3k_i/2 \leq \frac35(k-1).
\end{equation} 
By combining the last two inequalities, we have $\chi'(H_i) \leq |U\setminus U_i|$, and therefore, we are able to find such a proper edge-coloring, completing the proof.
\end{proof}

In order to establish Theorem \ref{main1}(ii), we modify the above proof. The improvement will come from picking the subsets $U_i \subset V_i$ more carefully (rather than arbitrarily), which will allow us to obtain a better bound on the chromatic index $\chi'(H_i)$ than that given by (\ref{chiindex}). We will need the following lemma, whose proof is given in the next section. 

% First, we replace the use of Lemma \ref{vizing}(i) by Lemma \ref{vizing}(ii), which tells us that the chromatic index of a multigraph $H$ is at most $\Delta+\mu$, where $\Delta$ is its maximum degree and $\mu$ is its edge multplicity. 

\begin{lemma}\label{lem:fix}
Let $k_i$ be a sufficiently large positive integer and $G[V_i]$ be a graph with $n_i \geq 2k_i-1$ vertices with minimum degree at least $k_i-1$.

Then there are a choice of $U_i\subset V_i$ with $|U_i|=k_i$ and one-to-one maps $$f_u:U_i \setminus (N_G(u) \cup \{u\}) \to W_i \cap N_G(u)$$ for $u \in U_i$ (where $W_i:=V_i \setminus U_i$) such that the chromatic index of the multigraph $H_i$, defined as in the proof of Theorem 
\ref{main1}(i), satisfies 
$$\chi'(H_i) \leq \left(\frac{9}{16}+o(1)\right)k_i.$$
Here, the $o(1)$ term is a function which tends to $0$ as $k_i \to \infty$. 
\end{lemma}

\begin{proof}[Proof of Theorem \ref{main1}(ii)] 
By Lemma \ref{lem:fix} above, we can choose $U_i$ and the one-to-one functions $f_u$ for $u \in U_i$ (rather than arbitrarily), to guarantee the following. For each $\varepsilon>0$ there is $C=C(\epsilon)$ such that if $k_i \geq C$, then the corresponding multigraph $H_i$ satisfies 
\begin{equation}\label{chibetter}
\chi'(H_i) \leq \left(\frac9{16}+\varepsilon\right)k_i.
\end{equation}
 
As in the proof of Theorem \ref{main1}(i), if $\chi'(H_i) \leq k-k_i$ for every $i$, then $G$ contains the desired weak immersion of $K_k$, and we are done. So, we may assume, for the sake of contradiction, that there is some $i$ for which $\chi'(H_i) > k-k_i$. If $k_i < C$, we also use the bound $\chi'(H_i) \leq 3k_i/2$ as in the proof of Theorem \ref{main1}(i) to obtain $k < 5C/2$. However, this contradicts our assumption that $k$ is sufficiently large. So, we may assume $k_i \geq C$ and, hence, $(9/16+\varepsilon)k_i \geq \chi'(H_i) > k-k_i$, which implies that $$k_i > k\big/\left(\frac{25}{16}+\varepsilon\right).$$ Similar to (\ref{firstinequality}), as $n_i \geq 2k_i - 1$, we have 
$$ k_i \leq n_i-k_i+1 = |W_i| + 1 \leq n-k+1.$$
If $k$ is sufficiently large, the previous two inequalities imply that $n > (1.64-\varepsilon)k$, contradicting the assumed bound on $n$. This completes the proof of Theorem \ref{main1}(ii). \end{proof}

To get the better upper bound on $\chi'(H_i)$ in Lemma \ref{lem:fix}, we pick the set $U_i$ by first including all vertices of $G[V_i]$ of degree at least some threshold $d_i$, and the remaining vertices we pick uniformly at random.\footnote{This is not quite accurate, as if there is more than $k_i-k_i^{0.9}$ vertices of degree at least $d_i$, we only pick $k_i-k_i^{0.9}$ of them to be in $U_i$ and pick an additional $k_i^{0.9}$ uniform random vertices to fill out the rest of $U_i$. Guaranteeing that there is still a reasonable amount of $U_i$ that is picked uniformly at random allows us to obtain that with probability $1-o(1)$ the edge multiplicity of $H_i$ is $o(k_i)$.} Instead of picking $U_i$ by this threshold, an alternative strategy is to add vertices to $U_i$ greedily one at a time that would have maximum expected degree in $H_i$ conditioned on the vertex otherwise being in $W_i$ and the remaining vertices picked in $U_i$ uniformly at random. We stop this greedy process if the expected maximum degree of $H_i$ is small enough, and add the remaining vertices of $U_i$ uniformly at random. We suspect this process would produce a set $U_i$ for which $\chi'(H_i)$ is a constant factor smaller, but analyzing this process appears to be more challenging. 

%\jacob{Add a comment on better constant factor.} 

\section{Proof of Lemma \ref{lem:fix}}~\label{sec3}
\label{betterchoices}

Before turning to the proof of Lemma \ref{lem:fix}, we need some preparation. We recall two well known probabilistic results from \cite{AlSp} which will be used in the sequel.

A \emph{martingale} is a sequence of random variables $X_0,X_1,\ldots,X_n$ that satisfies for any $0 \leq i < n$ that $\mathbf {E} [X_{i+1}\mid X_{1},\ldots ,X_{i}]=X_{i}$. 

\begin{lemma}\label{Azuma}
(i) (Chernoff's bound)  If $X$ is a sum of $n$ independent random variables taking values in $[0,1]$ and $\mu = \mathbb{E}[X]$, then for any $t>0$ we have $\mathbb{P}(X \geq \mu+t) < e^{-2t^2/n}$ and, similarly, $\mathbb{P}(X \leq \mu-t) < e^{-2t^2/n}$. 

(ii) (Azuma-Hoeffding bound) If $X_0,X_1,\ldots,X_n$ is a martingale which satisfies $|X_{i+1}-X_i| \leq c_i$ for each $0 \leq i <n$, then $\mathbb{P}(X_n-X_0>t) \leq  \exp \left(\frac{-t^{2}}{2\sum _{i=0}^{n-1}c_{i}^{2}}\right)$ and, similarly, $\mathbb{P}(X_n-X_0<-t) \leq  \exp \left(\frac{-t ^{2}}{2\sum _{i=0}^{n-1}c_{i}^{2}}\right)$.
\end{lemma}
If instead of (i), $X$ has a hypergeometric distribution, i.e., $X$ counts the number of successes when sampling from a population without replacement, then $X$ is at least as concentrated as the corresponding binomial distribution, and hence it also satisfies Chernoff's bound (Lemma~\ref{Azuma}(i)). 

We will also need some simple probabilistic statements about graphs.

\begin{lemma}\label{aux1} Let $n$ be a sufficiently large positive integer, $s\le n$,  and let $F$ be a graph on $n$ vertices. For any vertex $w\in V(F)$, let $d_w$ denote the degree of $w$ in $F$. Let $S$ be a uniform random subset of $s$ vertices of $F$. 

(i) For any vertex $w\in V(F)$, the expected number of neighbors of $w$ that belong to $S$ is $d_w s/n$. 

(ii) The probability that there is a vertex $w$ for which the number of neighbors of $w$ in $S$ deviates from this expectation by more than $\sqrt{n \ln n}$ is $o(1)$.

(iii) If $F$ has maximum degree $d$ such that $s\le d \leq n^{.9}$, then with probability $1-o(1)$, every vertex $w$ has $o(s)$ neighbors in $S$. (Here, $o(s)$ stands for a function tending to $0$ as $s \to \infty$.)
\end{lemma}

\begin{proof}
(i) Consider a vertex $w$ of $F$. For each vertex, the probability that it belongs to $S$ is $s/n$. So, by linearity of expectation, the expected number of neighbors of $w$ in $S$ is $d_w s/n$. 

(ii) The hypergeometric distribution is at least as concentrated as the corresponding binomial distribution (for a proof, see Section 6 of \cite{Hoeffding}). By Chernoff's bound (Lemma~\ref{Azuma}(i)), the probability that the number of neighbors of $w$ in $S$ deviates from its mean by more than $\sqrt{n\ln n}$ is at most $2e^{-2(\sqrt{n \ln n})^2/n}= 2n^{-2} = o(1/n)$. By the union bound, the probability that there is a vertex $w$ whose number of neighbors in $S$ exceeds its expected value is at most $n \cdot o(1/n)=o(1)$. 

(iii) Suppose that $F$ has maximum degree $d$ with $d \leq n^{.9}$, and let $r=\max(100,\lfloor 3s\sqrt{d/n}\rfloor)$. Each vertex $w$ has at most ${d \choose r}$ subsets of $r$ neighbors. The probability that a given vertex $w$ has at least $r$ neighbors in $S$ is at most 
$${d \choose r}{n-r \choose s-r}/{n \choose s} \leq (ed/r)^r(s/n)^r=\left(\frac{eds}{rn}\right)^r \leq (d/n)^{r/2}=o(1/n).$$
By the union bound over the $n$ vertices, with probability $1-o(1)$, every vertex has at most $r=o(s)$ neighbors in $S$. 
\end{proof}
\medskip

We now turn to the proof of Lemma \ref{lem:fix}. We need to show that for every $i$, one can choose $U_i \subset V_i$ and one-to-one functions $f_u$ such that the corresponding multigraph $H_i$ satisfies $$\chi'(H_i) \leq \left(\frac{9}{16}+o(1)\right)k_i.$$  Our proof is probabilistic. We fix an index $i$. To simplify the notation, in the rest of this section, we will suppress the index $i$. In particular, with no danger of confusion, we will write 
$n$, $k$,  $V = U\cup W$, and $H$ for $n_i$, $k_i$, $V_i = U_i \cup W_i$, and $H_i$,  respectively.

%\janos{I think that these two well known statements should be displayed as a two-part lemma. (i) Chernoff, (ii) Azum-Hoeffding with references to Alon-Spencer, say.} 

% The Chernoff bound in the form of Hoeffding's inequality says that if $X$ is a sum of $n$ independent random variables taking values in $[0,1]$ and $\mu = \mathbb{E}[X]$, then for any $t>0$ we have $\mathbb{P}(X \geq \mu+t) < e^{-2t^2/n}$ and similarly $\mathbb{P}(X \leq \mu-t) < e^{-2t^2/n}$. 

%\janos{Do we need this only in the positive direction? It should be remarked after lemma (ii)}\jacob{The setup of Azuma-Hoeffding is symmetric about $0$, so the negative direction has the same bound. I included both directions of both Chernoff and Azuma-Hoeffding, but I am inclined to delete the negative directions of both and just note the analogous negative direction holds by symmetry.}

\begin{proof}[Proof of Lemma \ref{lem:fix}]
Recall that for any $U \subset V$ with $|U|=k$, there is at least one one-to-one map $f_u$ from $U \setminus (N_G(u) \cup \{u\})$ to $W \cap N_G(u)$, since we have $|W \cap N(u)|-|U \setminus (N_G(u) \cup \{u\})|=\deg(u)-k+1 \geq 0$. The last inequality holds, because the the minimum degree of $G$ is at least $k-1$. 

In the next paragraph we describe how to choose the random set $U\subset V$, followed by how we pick the one-to-one functions $f_u$ for $u \in U$. Let $d:=9k/8=1.125k$. Let $\phi(k)=o(k)$ be such that $\phi(k)/k \to 0$ sufficiently slowly (for instance, $\phi(k)=k^{.9}$ will do).  
%\janos{I am not sure if $f$ is a good notation: it was used for the function $f_u$. Perhaps we can write $\phi(k)$, instead, without abbreviating it as $\phi$. }

If $G$ has less than $k-\phi(k)$ vertices of degree at least $d$, then let $U(d)$ be the set of vertices of degree at least $d$. In this case, all vertices not in $U(d)$ have degree less than $d$. Otherwise, there are at least $k-\phi(k)$ vertices of $G$ of degree at least $d$, and we let $U(d)$ be an arbitrary set of $k-\phi(k)$ vertices of $G$ of degree at least $d$. In either case, we let $\ell:=|U(d)|$, so $\ell \leq k-\phi(k)$. We let $U(d) \subset U$, picking the remaining $k-\ell \geq \phi(k)$ vertices of $U$ uniformly at random from the remaining $n-\ell$ vertices of $V$. %\janos{I changed the notation $U(d)$ to $U(d)$, because $U_i$ was used before.}

Having described the semi-random process for picking $U$, we next explain how to pick the one-to-one functions $f_u$ from $U \setminus (N_G(u) \cup \{u\})$ to $W \cap N_G(u)$, for each $u \in U$.   Suppose that there are vertices $w \in W $ and distinct $u,u' \in U$ with $u \in U(d)$ such that $u,u'$ are neighbors of $w$ and $u,u'$ are nonadjacent.  Then we set $f_u(u')=f_{u'}(u)=w$, partially defining $f_u$ and $f_{u'}$.  In order to define the rest of the mapping of $f_u$ from $U \setminus (N(u) \cup \{u,u'\})$ to $W \cap N(u)\setminus \{w\}$ (and the rest of the mapping $f_{u'}$), we modify our graph $G$ as follows. We delete edges $(u,w)$ and $(u',w)$  from $G$ and add the edge $(u,u')$ to obtain a new graph $G'$.  Notice that this operation, transforming $G$ to $G'$, did not change the degree of the two vertices $u,u' \in U$, but decreased by two the degree of the vertex $w \in W$. So, $\deg_{G'}(u) =\deg_{G}(u)$ for all $u \in U$ and $\deg_{G'}(w) \leq \deg_{G}(w)$ for all $w \in W$.  Moreover, we have

$$U \setminus (N_G(u) \cup \{u,u'\}) = U \setminus (N_{G'}(u) \cup \{u\})$$

\noindent and 

$$W \cap N_G(u)\setminus \{w\} = W \cap N_{G'}(u).$$

\noindent Thus, it now suffices to define the rest of the mapping $f_u$ from $U \setminus (N_{G'}(u) \cup \{u\})$ to $ W \cap N_{G'}(u)$ for each $u \in U$.

We keep repeating this operation, further (partially) defining the functions $f_u$ each time, until there is no such triple of vertices remaining.  At the end of this process, we obtain a new graph $G^{\ast}$ on $V$, and partially defined functions $f_u$ for each $u \in U$. Equivalently, defining the rest of each $f_u$ with respect to $G$ is the same as fully defining $f_u$ for each $u \in U$ with respect to the graph $G^*$. 
  
For each vertex $v$, let $$U_v:=N_{G^{\ast}}(v) \cap U.$$ We have $\deg_{G^{\ast}}(u) =\deg_{G}(u)$ for all $u \in U$, and $\deg_{G^{\ast}}(w) \leq \deg_{G}(w)$ for all $w \in W$. By the definition of $G^{\ast}$, for each $w \in W$ and $u \in U_w \cap U$, we have $U_w \setminus \{u\} \subset U_u$. Finally, we complete the definition of $f_u$, for each $u \in U$, by picking uniformly at random a one-to-one mapping from $U \setminus (N_{G^{\ast}}(u) \cup \{u\})$ to $W \cap N_{G^{\ast}}(u)$. 

Having described the semi-random process for picking $f_u$, we obtain the semi-random multigraph $H$.  Recall that $V(H) = W$ and for each nonadjacent pair $(u,u’)$ in $G$ with $u,u' \in U$ and $f_u(u') \not = f_{u'}(u)$, we add an edge between $f_u(u')$ and $f_{u'}(u)$ in $H$.  By Lemma \ref{vizing}(ii), the chromatic index of the multigraph $H$ satisfies $\chi'(H) \leq \Delta+\mu$. 
%It remains to analyze the maximum degree, $\Delta$, of our multigraph, and then its edge multiplicity, $\mu$. We will be able to guarantee that the edge multiplicity satisfies $\mu=o(k)$, so it will be a lower order term. 
Therefore, in order to complete the proof of Lemma~\ref{lem:fix}, it is sufficient to verify the following two statements.

\begin{proposition}
    \label{cl1}
With probability $1-o(1)$, the maximum degree $\Delta=\Delta(H)$ of the multigraph $H$ satisfies $$\Delta\le\left(\frac{9}{16}+o(1)\right)k,\;\; \mbox{ as}\;
\; k\rightarrow\infty.$$
\end{proposition}

\begin{proposition}\label{cl2}
With probability $1-o(1)$ the maximum multiplicity $\mu=\mu(H)$ of the edges of the multigraph $H$ satisfies 
$$\mu=o(k),\;\; \mbox{ as}\;\; k\rightarrow\infty.$$
\end{proposition}

In the above two propositions, the $o(1)$ terms tend to $0$ as $k \to \infty$. 

\begin{proof}[Proof of Proposition~\ref{cl1} (Analyzing $\Delta$)] 
Recall that $U(d)$ denotes the set of vertices of $G$ with degree at least $d$, and $|U(d)|=\ell.$
\smallskip

\textbf{Case I:} Suppose  $G$ has less than $k-\phi(k)$ vertices of degree at least $d$.  Then we have $\ell < k-\phi(k)$, where $\phi(k)=k^{.9}.$  In this case, we know that the vertices not in $U(d)$ have degree less than $d$, and this will only be used in establishing the upper bound on $|U_w|$ below. 

%\janos{I lost track: where do we use this assumption?} \jacob{Good question. We are using this assumption in the upper bound for (4) below, the second term is coming from the assumption that each vertex not in $U(d)$ has degree less than $d$.}

%First, we give an upper bound on the expected degree of each vertex in $H$. Next, we prove that the degree is highly concentrated around its expected value. A union bound then shows that the probability that there is a vertex whose degree sharply deviates from its expected degree is small. 

For each vertex $w \in V \setminus U(d)$, let $\ell_w$ be the number of neighbors of $w$ in $U(d)$. As $|U(d)|=\ell$, we have $0 \leq \ell_w \leq \ell < k-\phi(k) \leq k$. The proof easily follows Claims \ref{firstpropinproof}, \ref{secondpropinproof}, and \ref{thirdpropinproof}, whose proofs we defer for now.

\begin{claim}\label{firstpropinproof}
With probability $1-o(1)$, for each vertex $w \in V \setminus U(d)$, its number of neighbors in $U$ satisfies 
\begin{equation}\label{uwbound2} |U_w| \leq \ell_w+(d-\ell_w)\left(\frac{k-\ell}{2k-\ell-1} \right)+o(k).\end{equation}    
\end{claim}

\begin{claim}\label{secondpropinproof}
With probability $1-o(1)$, for each vertex $w \in W$, its degree in $H$ is at most $k/4$ or at most 
\begin{equation}\label{degreeexpect1}|U_w|-\ell_w\left(\frac{d-1-k}{d-1-|U_w|}\right)+o(k).
\end{equation}
\end{claim}

We use Claim \ref{secondpropinproof} to give an upper bound on the maximum degree of $H$. Note that also we have some constraints on the variables that appear in the upper bound on the degree of $w$. Namely, Claim \ref{firstpropinproof} gives the upper bound (\ref{uwbound2}) on $|U_w|$, and we already observed that $0 \leq \ell_w \leq \ell \leq k$.

It is convenient to parametrize by normalizing every variable by dividing it by $k$ and by ignoring the lower order terms, as they do not have an asymptotic effect. So with no danger of confusion, we delete all $o(1)$ terms after dividing by $k$.  In what follows, we compute the asymptotics as $k \to \infty$. So, set \[\alpha=\ell_w/k,\;\;\;\beta=\ell/k,\;\;\; \gamma=|U_w|/k,\;\;\; \delta=d/k=9/8.\] 
We need to check that the maximum degree of $H$ at most $(9/16+o(1))k$. We accomplish this by using Claim \ref{secondpropinproof} to check that each vertex $w$ of $H$ has degree at most $(9/16+o(1))k$. As $1/4 < 9/16$, each vertex $w$ of degree at most $k/4$ in $H$ is fine. For each vertex $w$ of degree more than $k/4$ in $H$, the upper bound (\ref{degreeexpect1}) on the degree of $w$ in $H$ becomes, asymptotically as $k \to \infty$, $k$ times 
\begin{equation}\label{degreeexpectgreek1}\gamma-\alpha\left(\frac{\delta-1}{\delta-\gamma}\right).
\end{equation}
The (asymptotic values of the) variables satisfy the constraints $0 \leq \alpha \leq \beta \leq 1$, and $0 \leq \gamma \leq \alpha+(\delta-\alpha)(1-\beta)/(2-\beta)$, with this last inequality coming from (\ref{uwbound2}). The following claim then completes the proof in this case. 

\begin{claim}\label{thirdpropinproof}
The maximum value of (\ref{degreeexpectgreek1}), under the conditions $0 \leq \alpha \leq \beta \leq 1$ and $0 \leq \gamma \leq \alpha+(\delta-\alpha)(1-\beta)/(2-\beta)$, is $9/16$. 
\end{claim}

\begin{proof}[Proof of Claim \ref{firstpropinproof}]
Recall that $G^{\ast}$ denotes the graph constructed from $G$ using the semi-random process described at the beginning of the proof of Lemma~\ref{lem:fix}. Let $F$ denote the subgraph of $G^{\ast}$ induced by the set $V \setminus U(d)$, so $F$ has $n-\ell$ vertices. Note that $S:=U \setminus U(d)$ is a uniformly random subset of $s=k-\ell$ vertices from $F$. By Lemma \ref{aux1}, the expected number of neighbors of $w$ in $S$ is at most $(d-\ell_w)s/(n-\ell)$, and with probability $1-o(1)$, no $w$ exceeds its expected number of neighbors in $S$ by more than $\sqrt{n\ln n}$. Hence, with probability $1-o(1)$, for each vertex $w \in V \setminus U(d)$, its number of neighbors in $U$ satisfies 
\begin{equation}\label{uwbound} |U_w| \leq \ell_w+(d-\ell_w)\left(\frac{k-\ell}{n-\ell} \right)+\sqrt{n\ln n}.\end{equation}
If we are in the case $n \leq d^{1.9}$, we have $\sqrt{n\log n}=o(k)$. Also observe that $n \geq 2k-1$ and the upper bound on $|U_w|$ in (\ref{uwbound}) is decreasing as a function of $n$ and, hence, is maximized if $n=2k-1$. Otherwise,  we have $n > d^{1.9}$, and by Lemma \ref{aux1}(iii) applied to $F$, with probability $1-o(1)$, every vertex $w \in V \setminus U(d)$ has $|U_w|\leq \ell_w+o(k)$ neighbors in $U$. In either case, the desired upper bound (\ref{uwbound2}) on $|U_w|$ holds with probability $1-o(1)$. %We can also obtain a lower bound on $|U_w|$. As each vertex has degree at least $k-1$,... \janos{A bound is missing here!} ... 

\end{proof}

\begin{proof}[Proof of Claim \ref{secondpropinproof}]
Having analyzed the effect of picking the semi-random set $U$ in the proof of \ref{firstpropinproof}, this proof analyzes the effect of picking the random functions $f_u$ after having already chosen $U$. For a given vertex $u\in U$ whose degree in $G^{\ast}$ is at least $d$, we can view picking the function $f_u$ as a two-step process. First, we pick a uniform random subset of $d-|N_{G^{\ast}}(u) \cap U|$ vertices from $W \cap N_{G^{\ast}}(u)$, and we reduce the range of $f_u$ to this random subset. In the second step, we pick the one-to-one function into this reduced range uniformly at random. After completing the first step for each vertex of $U$ of degree at least $d$, for every $u \in U$, the function $f_u$ has a reduced range of size at most $d$. Observe that, after performing the first step for each vertex $u$ in $U$, the number of vertices that can have degree at least $k/4$ in $H$ is at most $kd/(k/4)=4d$.  Let $T \subset W$ denote this set of at most $4d$ vertices. 

Let $u \in U(d)$ be a vertex adjacent to some $w\in V\setminus U(d)$. In $G^{\ast}$, $u$ is adjacent to all neighbors of $w$ in $U \setminus \{u\}$. Thus, $u$ has at least $|U_w|-1$ neighbors in $U$, and so at most $k-|U_w|$ vertices other than $u$ that are not adjacent to $u$. The probability that $w$ is in the image of $f_u$ is at most $\frac{k-|U_w|}{d-(|U_w|-1)}$. Hence, by linearity of expectation, the expected number of vertices $u \in N(w) \cap U(d)$ for which $w \in f_u$ is at most $\ell_w\left(\frac{k-|U_w|}{d+1-|U_w|}\right)$. By Chernoff's bound (Lemma ~\ref{Azuma}(i)), the probability that the number of vertices $u \in N(w) \cap U(d)$ for which $w \in f_u$ exceeds the expected number by at least $\sqrt{k\ln k}$ is at most $e^{-2(\sqrt{k \ln k})^2/d_1} = o(1/k)$. By the union bound over all the at most $4d <5k$ vertices $w \in T$,  with probability at least $1-5k\cdot o(1/k) = 1-o(1)$, every vertex $w$ of $H$ has degree at most $k/4$ or at most 
\begin{equation}\label{degreeexpect}|U_w|-\ell_w+ \ell_w\left(\frac{k-|U_w|}{d-1-|U_w|}\right)+\sqrt{k\ln k}=|U_w|-\ell_w\left(\frac{d-1-k}{d-1-|U_w|}\right)+\sqrt{k\ln k}.
\end{equation}
\end{proof}

\begin{proof}[Proof of Claim \ref{thirdpropinproof}]
Notice that the objective function (\ref{degreeexpectgreek1}) is decreasing in $\alpha$, and hence, given $\beta,\gamma,\delta$, is maximized when $\alpha$ is as small as possible given the constraints. This implies that $\alpha=0$ or 
$\gamma = \alpha+(\delta-\alpha)(1-\beta)/(2-\beta)$. 
\smallskip

{\bf Case 1.} $\alpha=0$. In this case, the maximum degree of $H$ is asymptotically $k\gamma \leq k\delta(1-\beta)/(2-\beta) \leq k\delta/2$, where the last inequality is tight for $\beta=0$. This settles this case. 
\smallskip

{\bf Case 2.} $\gamma = \alpha+(\delta-\alpha)(1-\beta)/(2-\beta)$. Let $\eta=1-(1-\beta)/(2-\beta)$, so $\gamma=\alpha+(\delta-\alpha)(1-\eta)=\eta \alpha+(1-\eta)\delta$ with $1/2 \leq \eta \leq 1$ and $\alpha \leq 2-1/\eta$. As a function of $\alpha$, the objective function after substituting in the formula $\gamma=\eta \alpha+(1-\eta)\delta$ is concave for each fixed $\eta$, and hence the objective function is maximized at the unique stationary value $\alpha^*=(9-3/\eta)/8$ (which is feasible if $\alpha^* \leq \beta$ or equivalently $\eta \geq 5/7$) or at the boundary point $\alpha=2-1/\eta$. 
\smallskip

{\bf Case 2a.} $\eta \geq 5/7$ and $\alpha=\alpha^*$. The objective function becomes $(3+1/\eta)/8$, which is decreasing for $\eta \geq 5/7$ and, hence, maximized at $\eta=5/7$. In this case, the maximum is $11/20 = 0.55 \leq 0.5625 = 9/16$. 
\smallskip

{\bf Case 2b.} $1/2 \leq \eta < 5/7$. In this case, the maximum is achieved at $\alpha=\beta$, that is, $\alpha=2-1/\eta$. Substituting in, the objective function becomes $$\frac{7\eta+1}{8} - \frac{2+1/\eta}{8-7\eta}.$$
The objective function is now decreasing in $\eta$ in its domain. Hence, it is maximized for $\eta=1/2$, and its value is at most $9/16$. 

\end{proof}

\textbf{Case II:}  Suppose $G$ has at least $k-\phi(k)$ vertices of degree at least $d$.  Then $\ell= k-\phi(k)$, where $\phi(k)=k^{.9}.$  
In this case, we follow the exact argument as in Case I. As vertices in $V \setminus U(d)$ do not necessarily have degree less than $d$, we do not have the bound (\ref{uwbound}) but alternatively $$\ell_w \leq |U_w| \leq \ell_w+\phi(k)$$ as $w$ will be adjacent to its $\ell_w$ neighbors in $U(d)$ and may be adjacent to any of the $\phi(k)$ vertices in $U \setminus U(d)$. Also, the bound (\ref{degreeexpect}) on the degree of $w$ in $H$ still holds, and we similarly arrive at the optimization problem of maximizing (\ref{degreeexpectgreek1}), except the constraints are now that $0 \leq \alpha \leq \beta=1-\phi(k)/k=1-o(1)$ and $0 \leq \gamma \leq \alpha+o(1)$. So, apart from the $o(1)$ terms, this becomes a special case (when $\beta=1$) of the optimization problem already studied in Case I. Hence, the analysis of this case reduces to that of Case I. 
\end{proof}

\begin{proof}[Proof of Proposition~\ref{cl2} (Analyzing $\mu$)] 
The analysis of $\mu$ below works the same in both Case I (when $\ell < k-\phi(k)$) and Case II (when $\ell = k-\phi(k)$). Hence, we will not treat the two cases separately. 

We will show that, with high probability, the edge multiplicity $\mu$ is at most $\phi(k)$, where $\phi(k)=k^{.9}$. 

As in the proof of Proposition \ref{cl1}, at the analysis of $\Delta$, after picking the random set $U$, at the selection of the random functions $f_u$ for vertices $u$ of degree at least $d$, we will ignore some edges. As before, this will effectively guarantee that for this purpose the degree of each vertex in $U$ is at most $d$. Note that then the number of vertices $w \in W$ adjacent to at least $\phi(k)$ vertices in $U$ is at most $kd/\phi(k) \leq 2k^{1.1}$. 

Each vertex of $G$ (and, hence, each vertex in $U$) has at least $k-1-\ell \geq \phi(k)-1$ neighbors in $V \setminus U(d)$. The set $W$ is a uniform random subset of size $n-k$ of the set $V \setminus U(d)$, which has size $n-\ell$. As $n \geq 2k-1$, $|U| = k$, and $\ell \geq 0$, we have $n-k \geq \frac{1}{2}(n-\ell)$. Hence, for each $u \in U$, the expected number of neighbors of $u$ in $W$ is at least $(\phi(k)-1)/2$. It follows from the Chernoff bound (Lemma~\ref{Azuma}(i)) and the union bound that, with probability $1-o(1)$ (as $k \to \infty$), each vertex in $U$ has at least $\phi(k)/4$ neighbors in $W$. Condition on this outcome. 

%Every vertex in $U(d)$ has at least $d-k \geq k/9$ neighbors in $W$. From the Chernoff bound and the union bound \jacob{ADD DETAILS}, with probability at least $3/4$, every vertex in $U \setminus U(d)$ has at least $(k-\ell)/4$ neighbors in $W$. Condition on this outcome. 

We will check that each pair $w,w'$ of distinct vertices in $W$ have edge multiplicity in $H$ at most $3\phi(k)$. Note that if a pair $w,w' \in W$ satisfies that $w$ has less than $\phi(k)$ neighbors in $U$, then the edge multiplicity in $H$ of the pair $w,w'$ is less than $\phi(k)$. 

Fix a pair $w,w' \in W$ of distinct vertices, each of which has at least $\phi(k)$ neighbors in $U$ (there are at most $(2k^{1.1})^2=4k^{2.2}$ such pairs). For each pair $u,u' \in U$, then, the probability that $f_u(u')=w$ and $f_{u'}(u)=w'$ is at most $(4/\phi(k))^2 $. Consider the random variable $Y$ which is the edge multiplicity of the pair $w,w'$ in $H$. The expected value of $Y$, by linearity of expectation over all pairs $u,u'$, satisfies $$\mathbb{E}[Y] \leq k^2(4/\phi(k))^2 =16k^{0.2}.$$

%If $u \in U(d)$ and $u' \in U \setminus U(d)$, then the probability that $f_u(u')=w$ and $f_{u'}(u)=w'$ is at most $(9/k)(4/(k-\ell))=36/k(k-\ell)$. If $u,u' \in U \setminus U(d)$, then the probability that $f_u(u')=w$ and $f_{u'}(u)=w'$ is at most $(4/(k-\ell))^2=16/(k-\ell)^2$. Consider the random variable $Y$ which is the edge multiplicity of the pair $w,w'$ in $H$. The expected value of $Y$, by linearity of expectation over all pairs $u,u'$, splitting into four cases depending on whether or not $u \in U(d)$ and whether or not $u' \in U \setminus U(d)$, satisfies $$\mathbb{E}[Y] \leq |U(d)|^2\frac{81}{k^2}+2|U(d)||U \setminus U(d)|\frac{36}{k(k-\ell)}+|U \setminus U(d)|^2\frac{16}{(k-\ell)^2} \leq 81+72+16 = 169.$$ 

As we reveal the functions $f_u$ one vertex $u$ at a time, $Y$ changes by at most 2 each time. Thus, by Azuma's inequality (Lemma~\ref{Azuma}(ii)), the probability that the pair $w,w'$ has edge multiplicity at least $100(k \log k)^{1/2}$ is at most $k^{-3}$. By the union bound over all the at most $4k^{2.2}$ pairs $w,w'$ of degree at least $\phi(k)$ in $U$, the probability the edge multiplicity $\mu$ exceeds $16k^{0.2}+100(k \log k)^{1/2} \leq \phi(k)$ is $o(1)$ (as $k \to \infty$). 
\smallskip

\end{proof}

Hence, by combining Lemma \ref{vizing} with Propositions \ref{cl1} and \ref{cl2}, we have

$$\chi'(H) \leq \left(\frac{9}{16} + o(1)\right)k.$$

\noindent This completes the proof of Lemma \ref{lem:fix}.\end{proof}

\section{Albertson's conjecture}~\label{sec4}
In this section, we prove Theorem \ref{main2}.  We recall an old conjecture of Hill, according to which the crossing number of the complete graph on $k$ vertices satisfies $cr(K_k)=H(k)$, where \[H(k):=\frac{1}{4} \left \lfloor \frac{k}{2}\right \rfloor \left \lfloor \frac{k-1}{2}\right \rfloor \left \lfloor \frac{k-2}{2}\right \rfloor \left \lfloor \frac{k-3}{2}\right \rfloor.\] It is known that $cr(K_k) \leq H(k)$ by a particular drawing of $K_k$. In the other direction, Balogh, Lidick\'y, and Salazar \cite{BLS} proved that $cr(K_k)$ is at least $0.9855 H(k)$, for large enough $k$. In particular, we have the following lemma.

\begin{lemma}[\cite{BLS}]\label{eq1234}
    If $k$ is sufficiently large, then $\cn(K_k) > k^4/65$.
\end{lemma}

A related old conjecture of Zarankiewicz, is that $\cn(K_{a,b})=\left \lfloor \frac{a}{2}\right \rfloor \left \lfloor \frac{a-1}{2}\right \rfloor \left \lfloor \frac{b}{2}\right \rfloor \left \lfloor \frac{b-1}{2}\right \rfloor$. Towards this conjecture, Balogh et al. \cite{BLNPSS} recently proved the following result. 

\begin{lemma}[\cite{BLNPSS}]\label{eq1234abc}
    If $a,b$ are sufficiently large, then $\cn(K_{a,b}) > 0.9118a^2b^2/16$.
\end{lemma}

Before turning to the proof of Theorem~\ref{main2}, as a warm-up, we establish the following useful asymptotic result.%, which will be useful for the proof of Theorem \ref{main2}.

\begin{lemma}\label{asympt}
For $\varepsilon > 0$ there is a $k_0$ such that the following holds.  If $G$ is a graph on $n < (1.64 - \varepsilon)k$ vertices with $\chi(G)=k$ and $k > k_0$, then \[\cn(G) > \cn(K_k) -   k^3 /2.\]  In particular, $\cn(G) \geq (1 - o(1))\cn(K_k)$, as $k\rightarrow\infty$.  
\end{lemma}

\begin{proof}
Let $G$ be drawn in the plane with $\cn(G)$ crossings. For a vertex $v$ of $G$, let $d(v)$ denote the degree of $v$ in $G$. By Theorem \ref{main1}(ii), $G$ contains $K_k$ as a weak immersion.  Let $v_1,\ldots, v_k$ be the branch vertices of the immersion, and let $P_{ij}$ be the path in $G$ used in the weak immersion with endpoints $v_i$ and $v_j$. 

Consider a drawing of $K_k$ in the plane, with vertices $v_1,\ldots, v_k$, where the edge between $v_i$ and $v_j$ is drawn along the path $P_{ij}$ such that it goes around every branch vertex that is an internal vertex of $P_{ij}$. By going either clockwise or counterclockwise around a branch vertex $v$, we can achieve that in the neighborhood of $v$, the drawing of the edge between $v_i$ and $v_j$ participates in at most $(d(v)-2)/2$ crossings. Apart from small neighborhoods of the branch vertices along the path $P_{ij}$, the drawing of the edge connecting $v_i$ to $v_j$ coincides with the drawing of $P_{ij}$. In particular, the drawing of the edge between $v_i$ and $v_j$ passes through every non-branch vertex that is an internal vertex of $P_{ij}$.

There are two types of crossings in this drawing of $K_k$. All crossings that are already crossings in the original drawing of $G$ are of \emph{type 1}, so there are at most $\cn(G)$ of them. The remaining crossings are of \emph{type 2}. They occur in small neighborhoods of vertices of $G$. The latter crossings fall into two categories depending on whether they occur in a small  neighborhood of a non-branch vertex or a branch vertex. 

For non-branch vertices $v$, the number of crossings in the drawing of $K_k$ at $v$ is at most ${f(v) \choose 2}$, where $f(v)$ denotes the number of paths $P_{ij}$ in the $K_k$ immersion, in which $v$ is an internal vertex of the path $P_{ij}$. Note that $f(v) \leq d(v)/2$, as $v$ is an 
internal vertex of at most $d(v)/2$ edge-disjoint paths. Obviously, $d(v) \leq n-1$ and there are at most $n-k$ non-branch vertices. Therefore, at non-branch vertices, the total number of crossings of type 2 is at most \[(n-k){(n-1)/2 \choose 2} \leq (n-k)n^2/8.\]

For each of the $k$ branch vertices $v_i$, there are $k-1$ paths $P_{ij}$ ending at $v_i$. Thus, $v_i$ is an internal vertex of at most $(d(v_i)-(k-1))/2 \leq (n-k)/2$ paths $P_{\ell j}$. In the drawing of $K_k$, the edge from $v_{\ell}$ to $v_j$ participates in at most $(d(v_i)-2)/2 < n/2$ crossings in the neighborhood of $v_i$, by going either clockwise or counterclockwise around $v_i$. Thus, in the neighborhoods of branch vertices, altogether there are  at most \[k\frac{n-k}{2}\frac{n}{2}\] crossings of type 2. 

Adding up the above two bounds and using our assumption that $n < (1.64 - \varepsilon)k$, we conclude that the total number of crossings in the drawing of $K_k$, which occur in small neighborhoods of the vertices of $G$ is at most $k^3/2$. Thus, we have produced a drawing of $K_k$ with fewer than $\cn(G)+k^3/2$ crossings. Consequently, we have \[\cn(G) > \cn(K_k) - k^3/2 = (1-o(1))\cn(K_k),\] as desired. \end{proof}

The well-known crossing lemma discovered by Ajtai, Chv\'atal, Newborn, Szemer\'edi \cite{ACNS} and independently, by Leighton \cite{L}, states that every graph $G$ with $n$ vertices and $m \geq 4n$ edges satisfies $\cn(G) \geq cm^3/n^2$, where $c>0$ is an absolute constant. The constant has been improved by several authors. The currently best constant is due to B\"ungener and Kaufmann.

\begin{lemma}[\cite{BK}]\label{ackcr}
    Let $G$ be a graph on $n$ vertices with $m$ edges.  If $m \geq 6.95n$, then $\cn(G) \geq\frac{1}{27.48}\frac{m^3}{n^2}.$
\end{lemma}

We also need the following two simple lemmas.

\begin{lemma}[\cite{PT13}]\label{addedge}
Let $G$ be a graph with $m$ edges, and let $x$ and $y$ be two nonadjacent vertices. Let $G+xy$ denote the graph obtained by adding the edge $(x,y)$. Then we have \[\cn(G+xy) \leq \cn(G)+m.\]
\end{lemma}

%\begin{lemma}\label{chromaticnumberchange} Let $G$ be a graph on $n$ vertices with chromatic number $k$ and $1 \leq a \leq n$ be an integer. Then $G$ has an induced subgraph on $a$ vertices that has chromatic number at least $\lfloor ka/n \rfloor$. \end{lemma}\begin{proof} There is a vertex partition of $G$ into $k$ independent sets. The union of any $t$ of these independent sets must have chromatic number $t$. Pick $a$ vertices by deleting the vertices from the largest independent sets. There are at least $\lfloor ka/n \rfloor$ remaining independent sets. \end{proof}

\begin{lemma}\label{numberofedgeschange}
Let $G$ be a graph on $n$ vertices with $m$ edges and $1 \leq a \leq n$ be an integer. Then $G$ has an induced subgraph on $a$ vertices with at least $m{a \choose 2}/{n \choose 2}$ edges. 
\end{lemma}
\begin{proof}
If we take a uniform random subset $A$ of $a$ vertices, the expected number of edges in $A$ is $m{a \choose 2}/{n \choose 2}$, and hence, there is an induced subgraph with at least that many edges. 
\end{proof}

% While the goal is to prove Albertson's conjecture, we can easily prove unconditionally that $\cn(G) \geq ck^4$ if $G$ has chromatic number $k \geq 5$. \todo{I am unsure what the best constant factor we know how to do here, but using a critical subgraph and the crossing lemma does reasonably well.} 

We are now ready to prove Theorem \ref{main2}. 

\medskip

\begin{proof}[Proof of Theorem \ref{main2}]
It suffices to prove the statement for $k$-critical graphs as every graph of chromatic number $k$ has a $k$-critical subgraph.  The proof is by induction on $\ell:=n-k$. The base case $\ell=n-k=0$ is trivial, as in this case $G=K_k$. Let $\ell$ be a positive integer and suppose we have established the desired result for smaller nonnegative integer values of $\ell$. Let $\varepsilon > 0$ and $k > 2^{70}$ be a large constant satisfying the conditions of Theorem \ref{main1}(ii), and let $G$ be a $k$-critical graph on $n < (1.64 - \varepsilon)k$ vertices with $n-k=\ell$.  By Lemma \ref{lemGa} (Gallai's theorem), $G$ has a vertex partition \[\mathcal{P}:V(G)=V_1 \cup \ldots \cup V_t\] into $t \geq 2$ nonempty parts such that each $G[V_i]$ is $k_i$-critical with  $n_i$ vertices with $n_i \geq 2k_i-1$, and every pair of vertices in different parts are adjacent. In particular, we have $\sum_{i=1}^t k_i=k$ and  $\sum_{i=1}^t n_i = n$. Set $\delta=1/11$, $\delta'=2^{-12}$, and $c=2^{-60}$, say. We distinguish three cases. 

\medskip

\noindent {\bf Case 1:} There is a part $V_i$ with $1<n_i \leq \delta k$.  Add missing edges to $V_i$, one at a time, until $V_i$ is complete. Each time we add an edge, we upper bound the increase of the crossing number by applying Lemma \ref{addedge}. As $G[V_i]$ is $k_i$-critical, each vertex in $G[V_i]$ has degree at least $k_i-1$. Hence, the number of nonadjacent pairs in $G[V_i]$ is at most $n_i(n_i-k_i)/2$. In total, by making $G[V_i]$ complete, we increase the crossing number by at most $n^2n_i(n_i-k_i)/4$. The resulting graph $G'$, obtained by completing part $V_i$, has $n$ vertices and is $k'$-critical with $k'=k+n_i-k_i$. Thus,  
\begin{eqnarray}\label{12x}
\cn(G') & \leq & \cn(G)+n^2n_i(n_i-k_i)/4 = \cn(G)+n^2n_i(k'-k)/4 \nonumber \\ 
& \leq &\cn(G)+\left(\frac{4}{65}-c\right)k^3(k'-k),
\end{eqnarray}
where in the last inequality we used that $n < (1.64 - \varepsilon)k$,  $n_i \leq k/11$, and $c= 2^{-60}.$

Applying the induction hypothesis to $G'$, we obtain 
\begin{equation}\label{12xy}
\cn(G') \geq \cn(K_{k'})+c(n-k')k'^3.\end{equation}

\noindent  Note that by averaging, we have 
\begin{eqnarray}\label{12xyz} \cn(K_{k'}) & \geq & \cn(K_k){k' \choose 4}/{k \choose 4} \geq (k'/k)^4\cn(K_k) \geq \left(1+4\left(\frac{k'}{k}-1\right)\right)\cn(K_k) \nonumber \\ & = & \cn(K_k) + 4\cn(K_k)(k'-k)/k \geq \cn(K_k) + \frac{4}{65}k^3(k'-k),
\end{eqnarray}
where in the last inequality we used Lemma \ref{eq1234}. 

Putting (\ref{12x})--(\ref{12xyz}) together, we obtain 
\begin{eqnarray*}
\cn(G) & \geq & \cn(K_{k'})+c(n-k')k'^3-\left(\frac{4}{65}-c\right)k^3(k'-k) \\ & \geq & 
\cn(K_k) + ck^3(k'-k)+c(n-k')k'^3\\ &\ge & \cn(K_k)+ck^3(n-k).\end{eqnarray*}
This completes the proof in this case. 

\medskip

\noindent \textbf{Case 2.} There is a part $V_i$ with $k_i \geq \delta'k$.  Applying Theorem~\ref{main1} to $G$, we obtain $K_{k}$ as a weak immersion in $G$.  By the proof of Lemma~\ref{asympt}, the number of crossings between the edges used in this weak immersion is at least 
\begin{equation}\label{jaj}
  \cn(K_{k})-k^3 /2.  
\end{equation}
 Furthermore, the proof of Theorem \ref{main1} shows that no matter how we partition part $V_i$ into $V_i=U_i \cup W_i$ with $|U_i|=k_i$, the edges in $G[W_i]$ are not used in the weak immersion. Note that $G[V_i]$ has minimum degree at least $k_i-1$, and hence, has at least $n_i(k_i-1)/2$ edges.  By Lemma \ref{numberofedgeschange}, we can pick this partition $V_i = W_i\cup U_i$ so that the number of edges in $G[W_i]$ is at least \[m_i:=\frac{1}{2}n_i(k_i-1){n_i-k_i \choose 2}/{n_i \choose 2}= \frac{1}{2}(k_i-1)(n_i-k_i)(n_i-k_i-1)/(n_i-1).\]
As all three numbers $k$, $k_i\ge \delta'k$, and $n_i-k_i \geq k_i-1$ are sufficiently large, we have $m_i \geq 6.95(n_i-k_i)$. Thus, we can apply Lemma \ref{ackcr} to obtain that \[cr(G[W_i]) \geq \frac{1}{27.48}\frac{m_i^3}{(n_i-k_i)^2} \geq 2^{-11}k_i^3(n_i-k_i).\]

\noindent We obtain that 
\begin{eqnarray*}\nonumber\cn(G) & \geq &\cn(K_{k})-k^3 /2+\cn(G[W_i])\\ &\geq & \cn(K_{k})-k^3 /2+2^{-11}k_i^3(n_i-k_i) \\
 &\geq & \cn(K_k)+c(n-k)k^3.\end{eqnarray*}
In the last inequality, we used that $k$ is sufficiently large, $k_i \geq \delta' k$, $n_i-k_i \geq k_i-1$, $n \leq (1.64 - \varepsilon)k$, and $c = 2^{-60}.$ This completes the proof in Case 2. 
\smallskip

\medskip

\noindent \textbf{Case 3.} Each part $V_i$ is either a singleton or satisfies $n_i > \delta k$ and $k_i < \delta'k$.  In this case, as $\ell=n-k>0$, we must have at least one part that is not a singleton. Recall that $n-k<0.64k$ and $n-k=\sum_i (n_i-k_i)$. For every $i$ for which $V_i$ is not a singleton, we have $n_i-k_i > \delta k - \delta'k = (\delta -\delta')k$. Thus, the number of parts that are not singletons is smaller than $0.64k/(\delta -\delta')k<8$, which implies that there are at most \emph{seven} non-singleton parts. 

Let $A$ be the union of the singleton parts and $B= V(G) \setminus A\neq\emptyset$. Since $B$ is the union of non-singleton parts $V_i$, each of which is larger than $\delta k$, we have $|B| > \delta k$. The chromatic number of $G$ is $k$. The chromatic number of $G[B]$, the subgraph of $G$ induced by $B$, is smaller than $7\delta'k$. Using that $A \cup B$ is a vertex partition of $G$, we obtain that the chromatic number of $G[A]$ is larger than $k-7\delta'k$. As $G[A]$ is a clique, we have $|A| > k-7\delta'k=(1-7\delta')k$.

It follows by averaging over all cliques of size $k$ in $K_{k+1}$, just like in (\ref{12xyz}),  that $\cn(K_k)/{k \choose 4}$ is a monotonically increasing function. Since it is bounded from above, it must converge. As $|A| > (1-7\delta')k$, we obtain that \[\cn(K_{|A|}) \geq \frac{\cn(K_k){|A| \choose 4}}{{k \choose 4}} \geq \left(\frac{|A|}{k}\right)^4\cn(K_k) \geq \left(1+4\left(\frac{|A|}{k}-1\right)\right)\cn(K_k) \geq   (1-28\delta')\cn(K_k),\] provided that $k$ is sufficiently large. 

Notice that the clique $G[A]$ and the complete bipartite graph between $A$ and $B$ are disjoint subgraphs of $G$. Therefore, we get 
\begin{eqnarray*}
\cn(G) & \geq & \cn(K_{|A|})+\cn(K_{|A|,|B|}) \geq (1-28\delta')\cn(K_k)+.9118|A|^2|B|^2/16 \\ & \geq & \cn(K_k)+\left(.9118(1-14\delta')\delta^2/16-28\delta'/64\right)k^4 \\
&\geq &\cn(K_k)+2^{-13}k^4 > \cn(K_k)+c(n-k)k^3.
\end{eqnarray*}
Here the second inequality follows by substituting in the bound from 
Lemma \ref{eq1234abc} on the crossing number of complete bipartite graphs and using the bound $\cn(K_k)\le k^4/64.$ The last inequality holds with $c=2^{-60}$, say, because $n-k<.64k.$  
\end{proof}

  \bigskip\noindent{\bf Acknowledgements.}
This material is based upon work supported by the National Science Foundation under Grant No. DMS-1928930, while the authors were in residence at the Simons Laufer Mathematical Sciences Institute in Berkeley, California, during the 2025 Extremal Combinatorics Program.  Parts of the solution to this optimization problem in Section \ref{betterchoices} were found with help from David Fox and Chat GPT-5. We would also like to thank Ji Zeng and Daniel Cranston for helpful comments related to an early draft of the paper.

\end{document}